\documentclass[11pt]{amsart}
\usepackage{graphicx,amssymb,amsmath,amsthm}
\usepackage{enumerate}
\usepackage{dsfont}
\usepackage[colorlinks, citecolor=red]{hyperref}
\usepackage{mathrsfs}
\usepackage{epsfig}
\usepackage{lscape}
\usepackage{subfigure}
\usepackage{epstopdf}

\usepackage{caption}
\usepackage{algorithm}
\usepackage{algpseudocode}

\usepackage{geometry}
\usepackage{amsfonts}
\usepackage{graphicx,cite}

\usepackage{paralist}
\usepackage{enumerate}
\usepackage{url}

\newcommand{\argmax}[1]{\mathop{\rm argmax}\limits_{#1}}

\newcommand{\cS}{{\mathcal S}}

\newcommand{\diag}{{\rm diag}}

\newtheorem{definition}{Definition}[section]
\newtheorem{corollary}[definition]{Corollary}

\newtheorem{theorem}[definition]{Theorem}
\newtheorem{lemma}[definition]{Lemma}

\newtheorem{remark}[definition]{Remark}

\newtheorem{example}[definition]{Example}

\date{}

\begin{document}
\baselineskip 18pt
\bibliographystyle{plain}
\title[Restricted invertibility with weighted columns]{A note on restricted invertibility with weighted columns}

\author{Jiaxin Xie}
\address{School of Mathematical Sciences, Beihang University, Beijing, 100191, China }
\email{xiejx@buaa.edu.cn}

\begin{abstract}
The restricted invertibility theorem was originally introduced by Bourgain and Tzafriri in $1987$ and has been considered as one of the most celebrated theorems in geometry and analysis.
In this note, we present weighted versions of this theorem with slightly better estimates.  Particularly, we show that for any $A\in\mathbb{R}^{n\times m}$ and $k,r\in\mathbb{N}$ with $k\leq r\leq \mbox{rank}(A)$, there exists a subset $\mathcal{S}$ of size $k$ such that $\sigma_{\min}(A_{\mathcal{S}}W_{\mathcal{S}})^2\geq \frac{(\sqrt{r}-\sqrt{k-1})^2}{\|W^{-1}\|_F^{2}}\cdot\frac{r}{\sum_{i=1}^{r}\sigma_{i}(A)^{-2}}$, where $W=\diag(w_1,\ldots,w_m)$ with $w_i$ being the weight of the $i$-th column of $A$. Our constructions are algorithmic and employ the interlacing families of polynomials developed by Marcus, Spielman, and Srivastava.
\end{abstract}
\maketitle

\section{Introduction}

\subsection{Restricted invertibility}

Given a matrix $A\in\mathbb{R}^{n\times m}$, the \emph{restricted invertibility problem} aims to find a subset $\cS\subseteq\{1,2,\ldots,m\}$, of cardinality $k\leq {\rm rank}(A)$ as large as possible, such that $$\|A_{\cS}x\|_2\geq c\|x\|_2$$ holds for all $x\in\mathbb{R}^{|\cS|}$ and to estimate the constant $c$.
Here we use  $A_{\cS}$ to denote the sub-matrix of $A$ obtained by extracting the columns of $A$ indexed by $\cS$, $\|\cdot\|_2$ denotes the Euclidean $2$-norm and $|\cS|$ denotes the cardinality of the set $\cS$.

In \cite{bourgain1987}, Bourgain and Tzafriri  provided the seminal result, known as the \emph{Bourgain-Tzafriri Restricted Invertibility Theorem}, to address the above problem. Their result has major influences on subsequent research, especially in Banach space theory and harmonic analysis
and recently has also had significant applications on problems in random matrix theory, applied mathematics, RIP-type results in compressed sensing and computer science (see \cite{casazza12,casazza09,chen2017,Litvak2019,Mohan1,Mohan2,
spielman17,Naor2017,spielman2012an,wen2018,xx2019,youssef2014,youssef2014-2,vershynin2001} and their references). The purpose of this note is to establish a weighted version of the restricted invertibility theorem by using the ingenious method of interlacing families of polynomials developed by Marcus, Spielman, and Srivastava  \cite{spielman151,spielman152}.

\subsection{Our contribution}
In this note, we consider a \emph{weighted version} of the restricted invertibility and aim to select multiples of the columns of the matrix.
Let $w_1,\ldots,w_m$ be any choices of multiples and denote $W=\diag(w_1,\ldots,w_m)$. For any $k\leq \mbox{rank}(A)$, we want to find a subset $\cS\subseteq\{1,2,\ldots,m\}$ of cardinality $k$, such that $\|A_{\cS}W_{\cS}x\|_2\geq c\|x\|_2$  and the constant $c$ is as large as possible.
This notion of selecting with weighted columns is useful and has been raised in many data analysis applications, especially in dimensionality reduction and linear sketching \cite{avron2013,Boutsidis2014,boutsidis2013,boutsidis2013l2,BSS14,woodruff2014}.

For convenience, throughout this note, we assume that $w_i\neq 0$ for $i=1,\ldots,m$.
We use $\|A\|_2$ and $\|A\|_F$ to denote, respectively, the operator norm and Frobenius norm of $A$. We denote the least singular value of $A$ as
$\sigma_{\min}(A):=\min\limits_{\|x\|_2=1}\|Ax\|_2$. We use $A^{\dagger}$ to denote the Moore-Penrose pseudo-inverse of $A$.

\subsubsection{Estimation in terms of rank}
Our main result is the following theorem.

\begin{theorem}
\label{main-theorem2}
Suppose that $A\in\mathbb{R}^{n\times m}$ and denote all its nonzero singular values as $\sigma_1(A)\geq\ldots \geq\sigma_{\emph{rank}(A)}(A)>0$.  Let $W\in\mathbb{R}^{m\times m}$ be a diagonal matrix whose diagonal elements are nonzero. Then for any fixed $k,r\in\mathbb{N}$ satisfying $k\leq r\leq\emph{rank}(A)$, there exists a subset $\cS\subseteq\{1,\ldots,m\}$ of size $k$ such that
$$\sigma_{\min}(A_{\cS}W_{\cS})^2\geq \frac{(\sqrt{r}-\sqrt{k-1})^2}{\|W^{-1}\|_F^{2}}\cdot\frac{r}{\sum_{i=1}^{r}\sigma_{i}(A)^{-2}}.$$
\end{theorem}

Note that our estimate is in terms of rank and a bound holds even when $k=\mbox{rank}(A)$. The proof of Theorem \ref{main-theorem2} provides a deterministic algorithm for computing the subset $\cS$ in $O\big(k(m-\frac{k}{2})n^{\theta+1}\big)$, where $\theta\in (2,2.373)$ is the matrix multiplication complexity exponent. We will introduce it in Section $4$.

Taking $W=I$ in Theorem \ref{main-theorem2}, we obtain the following corollary.
\begin{corollary}
\label{main-corollary}
Let $A$ be an $n\times m$ matrix. Then for any fixed $k,r\in\mathbb{N}$ satisfying $k\leq r\leq\emph{rank}(A)$, there exists a subset $\cS\subseteq\{1,\ldots,m\}$ of size $k$ such that
$$\sigma_{\min}(A_{\cS})^2\geq \frac{(\sqrt{r}-\sqrt{k-1})^2}{m}\cdot\frac{r}{\sum_{i=1}^r\sigma_{i}(A)^{-2}}.$$
Particularly, if $r=\emph{rank}(A)$, then
$$
\sigma_{\min}(A_{\cS})^2\geq \frac{(\sqrt{\emph{rank}(A)}-\sqrt{k-1})^2}{m}\cdot\frac{\emph{rank}(A)}{\|A^{\dagger}\|^2_F}.
$$
\end{corollary}

\begin{remark}
We next explain the reason why we employ an extra parameter $r$ in Theorem \ref{main-theorem2}, instead of taking $r=rank(A)$ directly. Generally, the parameter $r$ can be used as a thresholding rule for eliminating the smaller singular values of $A$. For example, consider a matrix $A\in\mathbb{R}^{n\times n}$ with $\sigma_i(A)=O(1)$ for $i=1,\ldots,r$ and $\sigma_i(A)=O(1/n^{1.5})$ for $i=r+1,\ldots,n$. Then $rank(A)=n$ and  for $k=O(r)$, Theorem \ref{main-theorem2} yields a subset $\cS$ of size $k$ for which $\sigma_{\min}(A_{\cS} W_{\cS})^2\geq O\big(\frac{n}{\|W^{-1}\|_F^2}\big)$, while by setting $r:=rank(A)$ Theorem \ref{main-theorem2} yields such a subset with $\sigma_{\min}(A_{\cS} W_{\cS})^2\geq O\big(\frac{1}{\|W^{-1}\|_F^2n}\big)$.
\end{remark}

\subsubsection{Estimation in terms of stable rank}
The estimation of the restricted invertibility principle is often stated in terms of stable rank in the literature \cite{bourgain1987,Mohan1,Mohan2,spielman17,youssef2014,youssef2014-2,vershynin2001,spielman2012an}. Here, we also present an estimation in terms of stable rank.
Define the Schatten $4$-norm stable rank as
 $$\mbox{srank}_4(A):=\frac{\big(\sum_{i}\sigma_i(A)^2\big)^2}{\sum_{i}\sigma_i(A)^4}.$$
Our result can be stated as follows.
\begin{theorem}
\label{main-theorem}
Let $A$ be an $n\times m$ matrix and $0<\epsilon<1$. Suppose $W\in\mathbb{R}^{m\times m}$ is a diagonal matrix whose diagonal elements are nonzero. Then there exists a subset $\cS\subseteq\{1,\ldots,m\}$ of size
$$
|\cS|=\big\lfloor(1-\epsilon)^2\emph{srank}_4(A)\big\rfloor+1
$$
such that
$$\sigma_{\min}(A_{\cS}W_{\cS})\geq \epsilon\frac{\|A\|_F}{\|W^{-1}\|_F}.$$
\end{theorem}

Let the columns of $A$ be the vectors $a_1,\ldots,a_m\in \mathbb{R}^n$, set $W=\diag(1/\|a_1\|_2,\ldots,1/\|a_m\|_2)$ in Theorem \ref{main-theorem}, we obtain the following ``normalized'' restricted invertibility principle:
\begin{corollary}
\label{main-theorem-col2}
Suppose $A$ is an $n\times m$ matrix and $k=(1-\epsilon)^2\emph{srank}_4(A)+1$ for some $\epsilon\in(0,1)$.  Then there exists a subset $\cS\subseteq\{1,\ldots,m\}$ of size $k$ such that
$$\sigma_{\min}(\tilde{A}_{\cS})\geq \epsilon,
$$
where $\tilde{A}$ denotes the matrix $A$ with normalized columns.
\end{corollary}

\begin{remark}
It is easy to see that Theorem \ref{main-theorem2} yields a wider range of the sampling parameter $k$ than that of Theorem \ref{main-theorem}. For example, it can be verified that if $\sigma_i(A)=O(1/\sqrt{i})$ for any $i\in\{1,\ldots,n\}$, then $\mbox{srank}_4(A)=O(\log^2 n)$ which may be much smaller than rank$(A)$ as $\mbox{rank}(A)=n$.
But there are also situations in which Theorem \ref{main-theorem} yields a better bound than that by Theorem \ref{main-theorem2}. Indeed, when $\sigma_{1}(A)=1$, $\sigma_{2}(A)=\ldots=\sigma_{m}(A)=O(\frac{1}{\sqrt{m}})$ and $k=(1-\epsilon)^2\mbox{srank}_4(A)+1$ for some $\epsilon\in(0,1)$. Then the bound provided by  Theorem \ref{main-theorem} is $\sigma_{\min}(A_SW_S)^2\geq O\big(\frac{\epsilon}{\|W^{-1}\|^2_F}\big)$, while in the same situation Theorem \ref{main-theorem2} yields the bound $\sigma_{\min}(A_SW_S)^2\geq O\big(\frac{1-\epsilon}{\|W^{-1}\|^2_Fr}\big)$.
Therefore, Theorem \ref{main-theorem2} and Theorem \ref{main-theorem} are independent of each other.
\end{remark}

\subsection{Related work}

Bourgian and Tzafriri \cite{bourgain1987} given the first result on the restricted invertibility, but only working with the square matrices. Later, Vershynin \cite{vershynin2001} extended their result to the case of rectangular matrices. Their proofs were based on a beautiful combination of probabilistic, combinatorial and analytic arguments. However, these proofs were non-constructive.

In \cite{spielman2012an}, Spielman and Srivastava provided a deterministic polynomial time algorithm to find the subset $\cS$ and improved the restricted invertibility of Bourgain-Tzafriri. Their proof used only basic linear algebra and can build the subset $\cS$ iteratively using a barrier potential function \cite{BSS14}.
Recently, Marcus, Spielman, and Srivastava \cite{spielman17} gave a different proof of such
result, using their powerful method of interlacing families \cite{spielman151,spielman152,spielmaniv}.
 Theorem $4.1$ in \cite{spielman17} shows that there exists a subset $\cS\subseteq\{1,\ldots,m\}$ of size $k\in(0,\mbox{srank}_4(A)\big]$ such that
$\sigma_{\min}(A_{\cS})^2\geq \frac{(\sqrt{\mbox{srank}_4(A)}-\sqrt{k})^2}{m}\cdot \frac{\|A\|_F^2}{\mbox{srank}_4(A)}.$
The algorithm in \cite{spielman17} runs in $O(kmn^{\theta+1})$ time.
By directly applying this method to hermitian matrices and their principal matrices, Ravichandran \cite{Mohan2} proved that for any $k\leq \mbox{srank}_4(A)$, there is a subset $\cS$ of size $k$ such that $\sigma_{\min}(A_{\cS})^2\geq\frac{\|A\|_2^2}{m}\big(\sqrt{1-\frac{k}{m}}
-\sqrt{\frac{k}{\mbox{srank}_{4}(A)}-\frac{k}{m}}\big)^2.$
In \cite{Naor2017}, Naor and Youssef also adopted the method of interlacing families of polynomials to consider the restricted invertibility problem.  Theorem $11$ in \cite{Naor2017} proved that there exists a subset $\cS$ of size $k<\mbox{rank}(A)$ such that
$\sigma_{\min}(A_{\cS})^2\geq \frac{(\sqrt{\mbox{rank}(A)}-\sqrt{k})^2}{m}\cdot\frac{\mbox{rank}(A)}{\|A^{\dagger}\|^2_F}.$

In \cite{youssef2014-2}, Youssef developed the weighted version of the restricted invertibility principle, adapting the techniques similar to \cite{BSS14,spielman2012an}. Theorem $1.1$ in \cite{youssef2014-2} proved that there exists a subset $\cS\subseteq\{1,\ldots,m\}$ of size
$
|\cS|\geq(1-\epsilon)^2\mbox{srank}_2(A)
$
such that
$\sigma_{\min}(A_{\cS}W_{\cS})\geq \epsilon\frac{\|A\|_F}{\|W^{-1}\|_F}.$
Here $\mbox{srank}_2(A):=\frac{\|A\|^2_F}{\|A\|^2_2}$ denotes the Schatten $2$-norm \emph{stable rank} of $A$.
 In \cite{youssef2014}, a normalized version of the restricted invertibility was studied.
 Theorem $3.1$ in \cite{youssef2014} proved that for any $\epsilon\in(0,1)$, there exists a subset $\cS\subseteq\{1,\ldots,m\}$ of size
$
|\cS|\geq (1-\epsilon)^2\mbox{srank}_2(A)
$
such that
$\frac{\epsilon}{2-\epsilon}\leq\sigma_{\min}(\tilde{A}_{\cS})\leq \sigma_{\max}(\tilde{A}_{\cS})\leq \frac{2-\epsilon}{\epsilon}$, where $\tilde{A}$ denotes the matrix $A$ with normalized columns.
Theorem \ref{main-theorem} and Corollary \ref{main-theorem-col2} are strict improvements on the above results, since they are available for a wider range of $k$. Here, we use the fact that
$
\mbox{srank}_4(A)=\frac{\big(\sum_{i}\sigma_i(A)^2\big)^2}{\sum_{i}\sigma_i(A)^4}\geq \frac{\big(\sum_{i}\sigma_i(A)^2\big)^2}{\sigma_1(A)^2\sum_{i}\sigma_i(A)^2}=\mbox{srank}_2(A)
$. Furthermore, if $A$ has many moderately large singular values, the above inequality can be far from tight. In addition, the bounds in Corollary \ref{main-theorem-col2} is tighter than that in  \cite{youssef2014}, as $\epsilon>\frac{\epsilon}{2-\epsilon}$ for $\epsilon\in(0,1)$.

In \cite{Naor2017}, Naor and Youssef provided an weighted version of the restricted invertibility. More precisely, combining Lemma $18$ and Theorem $9$ in \cite{Naor2017}, we can show that for any $k,r\in\mathbb{N}$ satisfying $k<r\leq \mbox{rank}(A)$, there exists a subset $\cS\subseteq\{1,\ldots,m\}$ of size $k$ and a universal constant $c>0$ such that
 \begin{equation}
 \label{rank-result}
 \sigma_{\min}(A_{\cS}W_{\cS})^2\geq c \cdot \frac{(r-k)\sum_{i=r}^{\mbox{rank}(A)}\sigma_i(A)^2}{\|W^{-1}\|^2_F\cdot r}.
 \end{equation}
 Their proof used a variety of deep tools from geometric functional analysis but was non-constructive. However, our proof could result in a deterministic algorithm. To the best of our knowledge, our algorithm is the first polynomial time algorithm for the weighted invertibility theorem in terms of rank. In addition, Theorem \ref{main-theorem2} can deal with full-rank selection while \eqref{rank-result} only deals with the case $k<\mbox{rank}(A)$. Besides, using the threshold parameter $r$, Theorem \ref{main-theorem} can reduce the effect caused by the smaller singular values of $A$, for which we have the following example.
\begin{example}
Suppose that the singular values of $A$ are $\sigma_{i}(A)=O(\sqrt{m-i+1})$ for $i=1,\ldots, m$ and let the sampling parameter $k=m-1$.
Then \eqref{rank-result} yields a subset $\cS$ of size $k$ for which $\sigma_{\min}(A_{\cS}W_{\cS})^2\geq O\big(\frac{1}{\|W^{-1}\|_F^2m}\big)$, while Theorem \ref{main-theorem2} yields such a subset with $\sigma_{\min}(A_{\cS}W_{\cS})^2\geq O\big(\frac{1}{\|W^{-1}\|^2_F\log m}\big)$.
\end{example}

 Very recently, in \cite{xx2019} the authors considered the problem of subset selection for matrices where their result is (only)  available for the case $AA^T=I$ and $k = \mbox{rank}(A)$. The infinite dimensional restricted invertibility has also been considered in the literature.
In \cite{casazza12}, Casazza and Pfander gave the definition for infinite dimensional restricted invertibility based on the notion of density from frame theory, and then they prove a infinite dimensional restricted invertibility theorem for $\ell_1$-localizd operators on arbitrary Hilbert spaces.

\subsection{Organization}
The paper is organized as follows.  In Section \ref{section-pre}, we will introduce
some notations and useful lemmas. We present the proof of Theorem \ref{main-theorem2} in Section \ref{section-proof}. In Section 4, we finally provide a deterministic selection algorithm for computing the subset
$\cS$ in Theorems \ref{main-theorem2} and \ref{main-theorem}.

\section{Preliminaries}
\label{section-pre}
\subsection{Notations and lemmas}
\label{section-pre1}

We use $\partial_{x}$ to denote  the operator that performs differentiation with respect to $x$. We say that a univariate polynomial is \emph{real-rooted} if all of its coefficients and roots are real. For a real-rooted polynomial $p$, we let $\lambda_{\min}(p)$  denote the smallest root of $p$ and we use $\lambda_{\ell}(p)$ to denote the $\ell$-th largest root of $p$.  We use $\mathbb{P}$ and $\mathbb{E}$  to denote the probability of an event and expectation of a random variable, respectively.

The following  inequality can help us to estimate the lower bound of the sum of a certain convex function.
\begin{lemma}
\label{lemma-Jensen}
Let $f$ be a function from $\mathbb{R}^n$ to $(-\infty,+\infty]$. Then $f$ is convex if and only if
$$
f(\mu_1x_1+\cdots+\mu_mx_m)\leq \mu_1f(x_1)+\cdots+\mu_mf(x_m)
$$
whenever $\mu_1\geq0,\ldots,\mu_m\geq 0, \mu_1+\cdots+\mu_m=1$.
\end{lemma}

We also need the following lemma.
\begin{lemma}[\cite{spielman152}, Lemma $4.2$]
\label{lem-p1}
For every square matrix $A$ and random vector ${\bf r}$,
$$
\mathbb{E}\det\big[A-{\bf rr}^T\big]=(1-\partial_t)\det\big[A+t\mathbb{E}{\bf rr}^T\big]\big|_{t=0}.
$$
\end{lemma}


\subsection{Interlacing families}
Our proof of Theorem \ref{main-theorem} builds on the method of interlacing families which is a powerful technology developed in \cite{spielman151,spielman152} by Marcus, Spielman and Srivastava in work of the solution to the Kadison-Singer problem.

Let $g(x)=\alpha_0 \prod\limits_{i=1}^{n-1}(x-\alpha_i)$ and $f(x)=\beta_0\prod\limits_{i=1}^{n}(x-\beta_i)$ be two real-rooted polynomials. We say $g$ interlaces $f$ if
$$
\beta_1\leq \alpha_1\leq\beta_2\leq\alpha_2\cdots\leq\alpha_{n-1}\leq \beta_n.
$$
We say that polynomials $f_1,\ldots,f_k$ have a \emph{common interlacing} if there is a polynomial $g$ so that $g$ interlaces $f_i$ for each $i$.

Following \cite{spielman152}\footnote{One may refer to \cite{spielman17,xx2019} for a more general definition of the interlacing families.}, we define the notion of an interlacing family of polynomials as follows.
Let $\cS_1,\ldots,\cS_m$ be finite sets, and for every assignment $s_1,\ldots,s_m\in \cS_1\times\cdots\times \cS_m$, let $f_{s_1,\ldots,s_m}(x)$ be a real-rooted degree $n$ polynomial with positive leading coefficient. For a partial assignment $s_1,\ldots,s_k\in \cS_1\times\cdots\times \cS_k$ with $k<m$, define
$$
f_{s_1,\ldots,s_k}:=\sum_{s_{k+1}\in \cS_{k+1},\ldots,s_m\in \cS_m}f_{s_1,\ldots,s_k,s_{k+1},\ldots,s_m}
$$
as well as
$$
f_{\emptyset}:=\sum_{s_{1}\in \cS_{1},\ldots,s_m\in \cS_m}f_{s_1,\ldots,s_m}.
$$

We say the polynomials $\{f_{s_1,\ldots,s_m}\}$ form an \emph{interlacing family} if for all $k=0,\ldots,m-1$ and all $s_1,\ldots,s_k\in \cS_1\times\ldots\times \cS_k$, the polynomials
$
\{f_{s_1,\ldots,s_k,t}\}_{t\in \cS_{k+1}}
$
have a common interlacing.

\begin{figure}[H]
\centering
\includegraphics[width=0.5\textwidth]{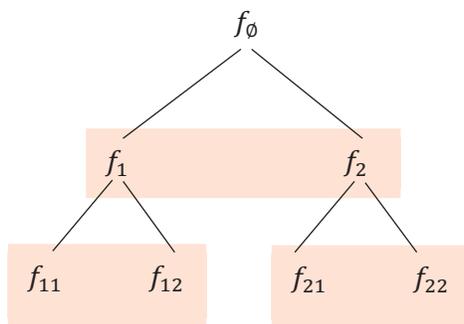}
\caption{Let $\cS_1=\{1,2\}$ and $\cS_2=\{1,2\}$. The polynomials $\{f_{ij}\}_{1\leq i,j\leq 2}$ form an interlacing family. The orange blocks denote subsets of polynomials that have a common interlacing. For every fixed $i$ $(i=\emptyset,1,2)$, each polynomial $f_{i}$ is a summation of the polynomials $\{f_{ij}\}_{j\in\{1,2\}}$. }
\label{figure:1}
\end{figure}

The following lemma which is proved in \cite[Theorem $2.7$]{spielman17} shows the utility of forming an interlacing family.
\begin{lemma}[\cite{spielman17}, Theorem $2.7$]
\label{root-interlacing}
Let $\cS_1,\ldots,\cS_m$ be finite sets, and let $\{f_{s_1,\ldots,s_m}\}$ be an interlacing family of degree $n$ polynomials. Then for all indices $1\leq j\leq n$, there exists some $a_1,\ldots,a_m\in \cS_1\times\ldots\times \cS_m$ and $b_1,\ldots,b_m\in \cS_1\times\ldots\times \cS_m$ such that
$$
\lambda_{j}(f_{a_1,\ldots,a_m})\geq \lambda_{j}(f_{\emptyset})\geq \lambda_{j}(f_{b_1,\ldots,b_m}).
$$
\end{lemma}

Let $u_1,\ldots,u_m$ be independent random vectors in $\mathbb{R}^n$ with finite support. Let $\ell_i$ be the size of the support of the random vector $u_i$, and let $u_i$ take the values $v_{i,1},\ldots,v_{i,\ell_i}$ with probabilities $p_{i,1},\ldots,p_{i,\ell_i}$. For $j_1\in[\ell_1],\ldots,j_m\in[\ell_m]$, define
\begin{equation*}
\label{fs1s2}
f_{j_1,\ldots,j_m}(x):=\bigg(\prod\limits_{i=1}^m p_{i,j_i}\bigg)\det\bigg[xI-\sum\limits_{i=1}^k v_{i,j_i}v_{i,j_i}^T\bigg].
\end{equation*}

\begin{lemma}[\cite{spielman152}, Theorem $4.5$]
\label{out-interlacing}
The polynomials $\{f_{j_1,\ldots,j_m}(x)\}$ form an interlacing family.
\end{lemma}

\subsection{Lower barrier function}
\label{section-pre3}
In this subsection we introduce the lower barrier potential function from \cite{BSS14,spielman152}.
For a real-rooted polynomial $p(x)$,
one can use the evolution of such barrier function to track the approximation locations
of the roots of $(1-t\partial_{x})p(x)$ where $t>0$.

\begin{definition}
For a real-rooted polynomial $p(x)$ with roots $\lambda_1,\ldots,\lambda_n$, define the lower barrier function of $p(x)$ as
$$
\Phi_{p}(x):=-\frac{p'(x)}{p(x)}=\sum\limits_{i=1}^n\frac{1}{\lambda_i-x}.
$$
\end{definition}

We have the following  lemma for the lower barrier function.
\begin{lemma}[\cite{spielman17}, Lemma $4.3$]
\label{lemma-lbf}
Let $p(x)$ be a real-rooted polynomial. Suppose that $b<\lambda_{\min}(p(x))$ and $\alpha>0$ satisfying
$$
\Phi_{p}(b)\leq \alpha.
$$
Then for any $t>0$ and $\delta:=\frac{t}{1+t\alpha}$, we have $b+\delta< \lambda_{\min}\big((1-t\partial_x)p\big)$ and
$$
\Phi_{(1-t\partial_x)p}(b+\delta)\leq \Phi_{p}(b).
$$
\end{lemma}

\section{Proof of Theorems \ref{main-theorem2} and \ref{main-theorem}}
\label{section-proof}
In this section we give the proof of the main results. The proof consists of two main parts. Firstly, choosing the subset $\cS$ with certain probabilities, we show that the characteristic
polynomials of the related sub-matrices form an interlacing family, and then we present an expression for the expected characteristic polynomial (the summation of the polynomials in the family). Secondly, we use the barrier function argument to establish a lower bound on the $k$-th largest zero of the expected characteristic polynomial.

Suppose that
$$
A=[a_1,\ldots,a_m]\in\mathbb{R}^{n\times m} \ \ \mbox{and} \ \ W=\diag(w_1,\ldots,w_m).
$$
We randomly pick up a column of $A$ with the probability proportional to the inverse squares of the weights, that is
$$
\mathbb{P}\{\mbox{the $i$th column is selected}\}=p_i:=\frac{w_i^{-2}}{\|W^{-1}\|^2_F}.
$$
For any subset $\cS=(s_1,\ldots,s_k)\in[m]^k$, where $[m]:=\{1,\ldots,m\}$, we set
\begin{equation}
\label{def-f}
f_{s_1,\ldots,s_k}(x):=\bigg(\prod\limits_{i=1}^k p_{s_i}\bigg)\det\bigg[xI-\sum\limits_{i=1}^k w_{s_i}^2a_{s_i}a_{s_i}^T\bigg].
\end{equation}
It can be see that
$
\sigma_{\min}(A_{\cS}W_{\cS})^2=\lambda_k \big(f_{\cS}(x)\big),
$
as
$$
\sigma_{\min}(A_{\cS}W_{\cS})^2=\lambda_{\min}\big(W_{\cS}^TA_{\cS}^TA_{\cS}W_{\cS}\big)=\lambda_{k}\big(A_{\cS}W_{\cS}W_{\cS}^TA_{\cS}^T\big)=\lambda_{k}\bigg( \sum\limits_{i\in\cS}w_{i}^2a_{i}a_{i}^T \bigg).
$$
By Lemma \ref{out-interlacing}, we have the following lemma.
\begin{lemma}
\label{if-def}
The polynomials $\{f_{s_1,\ldots,s_k}(x)\}$ defined by \eqref{def-f} form an interlacing family.
\end{lemma}

Construct an associated polynomial $f_{\emptyset}(x)$ called  \emph{expected characteristic polynomial} as
$$
f_{\emptyset}(x):=\mathbb{E}\det\bigg[xI-\sum\limits_{i=1}^k w_{s_i}^2a_{s_i}a_{s_i}^T\bigg]=\sum\limits_{s_1,\ldots,s_k\in[m]^k} f_{s_1,\ldots,s_k}(x).
$$
Lemma \ref{root-interlacing} tells us that every interlacing family of polynomials $\{f_{\cS}\}$ provides a associated polynomial  $f_{\emptyset}(x)$ with the property that there always exists a polynomial $f_{\widehat{\cS}}$ such that
$$
\lambda_{k}\big(f_{\widehat{\cS}}(x)\big)\geq \lambda_{k}\big(f_{\emptyset}(x)\big).
$$
As a result, if one can lower bound the $k$-th largest root of the polynomial  $f_{\emptyset}(x)$, then we claim that there exists a polynomial in the family having a $k$-th largest root which satisfies the same lower bound. The following lemma gives a formula for $f_{\emptyset}(x)$.
\begin{lemma}
\label{lemma-fix}
$f_{\emptyset}(x)$ has the following formula
$$f_{\emptyset}(x)=x^{n-k}\prod\limits_{i=1}^{\emph{rank}(A)}\bigg(1-\frac{\sigma_i(A)^2}{\|W^{-1}\|^2_F}\partial_x\bigg) x^{k},$$
where $\sigma_i(A),i=1,\ldots,\emph{rank}(A)$ are the singular values of $A$.
\end{lemma}
\begin{proof}
By introducing variables $z_1,\ldots,z_k$ and applying Lemma \ref{lem-p1} $k$ times, we obtain
$$
\begin{array}{ll}
f_{\emptyset}(x)&=\mathbb{E}\det\bigg[ xI-\sum\limits_{i=1}^k w_{s_i}^2a_{s_i}a_{s_i}^T\bigg]
\\[5mm]
&=\bigg(\prod\limits_{i=1}^k (1-\partial_{z_i})\bigg)\det \bigg[xI+\sum\limits_{i=1}^kz_i\mathbb{E}w_{s_i}^2a_{s_i}a_{s_i}^T\bigg]\bigg|_{z_1=\cdots=z_k=0}
\\[5mm]
&=\bigg(\prod\limits_{i=1}^k (1-\partial_{z_i})\bigg)\det \bigg[xI+\sum\limits_{i=1}^k\frac{z_i}{\|W^{-1}\|^2_F}AA^T\bigg]\bigg|_{z_1=\cdots=z_k=0}
\\[5mm]
&=\big(1-\partial_{z}\big)^k \det \bigg[xI+\frac{z}{\|W^{-1}\|^2_F}AA^T\bigg]\bigg|_{z=0},
\end{array}
$$
where the last equality based on the observation that the second determinant is a polynomial in $z:=z_{1}+\ldots+z_k$, and since for any differentiable function of $z$ we have $\partial_z=\partial_{z_i}$ for every $i=1,\ldots,k$, and the operator $1-\partial_z$ preserves the property of being a polynomial in $z$. Write $AA^T=U\Sigma U^T$, where $U$ is an orthogonal matrix and $\Sigma$ is a diagonal matrix whose diagonal equals $\big(\sigma_1(A)^2,\ldots,\sigma_n(A)^2\big)$. Then for every $x,z\in\mathbb{R}$ we have
$$
\begin{array}{ll}
\det \bigg[xI+\frac{z}{\|W^{-1}\|^2_F}AA^T\bigg]&=\det\bigg[U^T\bigg(xI+\frac{z}{\|W^{-1}\|^2_F}\Sigma\bigg)U\bigg]
\\[5mm]
&=\prod\limits_{i=1}^n\bigg(x+\frac{z}{\|W^{-1}\|^2_F}\sigma_i(A)^2\bigg)
\\[5mm]
&=x^{n-\mbox{rank}(A)}\prod\limits_{i=1}^{\mbox{rank}(A)}\bigg(x+\frac{z}{\|W^{-1}\|^2_F}\sigma_i(A)^2\bigg),
\end{array}
$$
where we used the fact that $\sigma_i(A)=0$ when $i\geq \mbox{rank(A)}$. Consequently,
$$
\begin{array}{ll}
f_{\emptyset}(x)&= x^{n-\mbox{rank}(A)} \big(1-\partial_{z}\big)^k \prod\limits_{i=1}^{\mbox{rank}(A)}\bigg(x+\frac{z}{\|W^{-1}\|^2_F}\sigma_i(A)^2\bigg)\bigg|_{z=0}
\\[5mm]
&=x^{n-k}\prod\limits_{i=1}^{\mbox{rank}(A)}\bigg(1-\frac{\sigma_i(A)^2}{\|W^{-1}\|^2_F}\partial_x\bigg) x^{k},
\end{array}
$$
where one can check the coefficients of polynomials that appear in the right hand sides of the first and the second equalities are equal to each other.
\end{proof}

Now we are ready to state the proof of Theorem \ref{main-theorem2}.
\begin{proof}[Proof of Theorem \ref{main-theorem2}]
Let
$$
g(x):=\prod\limits_{i=1}^{\mbox{rank}(A)}\bigg(1-\frac{\sigma_{i}(A)^2}{\|W^{-1}\|^2_F}\partial_x\bigg)x^k.
$$
For any $k\leq r \leq\mbox{rank}(A)$, we claim that
\begin{equation}
\label{proof-main2-1}
\lambda_{\min}(g)\geq\frac{(\sqrt{r}-\sqrt{k-1})^2}{\|W^{-1}\|_F^{2}}\cdot
\frac{r}{\sum_{i=1}^{r}\sigma_{i}(A)^{-2}}.
\end{equation}
From Lemma \ref{lemma-fix}, we know that
 $$\lambda_{k}(f_{\emptyset})=\lambda_{\min}(g),$$
which together with Lemmas \ref{root-interlacing} and \ref{if-def}, we know that there exists a sequence $s_1,\ldots,s_k\in[m]^k$ such that
\begin{equation}
\label{lambda-k}
\lambda_{k}(f_{s_1,\ldots,s_k})\geq\lambda_{\min}(g)\geq\frac{(\sqrt{r}-\sqrt{k-1})^2}{\|W^{-1}\|_F^{2}}\cdot
\frac{r}{\sum_{i=1}^{r}\sigma_{i}(A)^{-2}}.
\end{equation}
Recall that $f_{s_1,\ldots,s_k}$ is the characteristic polynomial of $\sum\limits_{i=1}^k w_{s_i}^2a_{s_i}a^T_{s_i}$ and since $k\leq r\leq\mbox{rank}(A)$, from \eqref{lambda-k} we know that $\lambda_{k}\big(\sum_{i=1}^k w_{s_i}^2a_{s_i}a^T_{s_i}\big)>0$. Thus the matrix $\sum\limits_{i=1}^k w_{s_i}^2a_{s_i}a^T_{s_i}$ must have rank $k$ and this implies that the sequence $s_1,\ldots,s_k$ consists of distinct elements. So, we conclude that there exists a subset $\cS\subseteq [m]$ of size $k$ for which
$$
\sigma_{\min}(A_{\cS}W_{\cS})\geq\frac{(\sqrt{r}-\sqrt{k-1})^2}{\|W^{-1}\|_F^{2}}\cdot
\frac{r}{\sum_{i=1}^{r}\sigma_{i}(A)^{-2}}.
$$
Now we remain to prove \eqref{proof-main2-1}. Let $p(x)=x^k$. For any $\alpha>0$, let $b=-k\alpha^{-1}$. Then
$$
\Phi_p(b)\leq\alpha.
$$
For any $i=1,\ldots,\mbox{rank}(A)$, set $$\delta_i=\frac{1}{\|W^{-1}\|^2_F\sigma_i(A)^{-2}+\alpha}.$$
Then applying Lemma \ref{lemma-lbf} $k$ times, we obtain
$$
\Phi_g\big(b+\sum\nolimits_{ i=1}^{\mbox{rank}(A)}\delta_i\big)\leq \Phi_p(b)\leq \alpha.
$$
Using the definition of the lower barrier function, we have
$$
\frac{1}{\lambda_{\min}(g)- \big(b+\sum\nolimits_{ i=1}^{\mbox{rank}(A)}\delta_i  \big)}\leq\Phi_g\big(b+\sum\nolimits_{ i=1}^{\mbox{rank}(A)}\delta_i\big)\leq \Phi_p(b)\leq \alpha.
$$
Hence
$$
\lambda_{\min}(g)\geq b+\sum\limits_{ i=1}^{\mbox{rank}(A)}\delta_i+\alpha^{-1}
=-k\alpha^{-1}+\sum\limits_{ i=1}^{\mbox{rank}(A)}\frac{1}{\|W^{-1}\|^2_F\sigma_i(A)^{-2}+\alpha}+\alpha^{-1},
$$
i.e.,
\begin{equation}
\label{proof-111}
\lambda_{\min}(g)\geq \frac{-(k-1)}{\alpha}+\sum\limits_{ i=1}^{\mbox{rank}(A)}\frac{1}{\|W^{-1}\|^2_F\sigma_i(A)^{-2}+\alpha}.
\end{equation}
Note that the function $x\to \frac{1}{\alpha+x}$ on $[0,+\infty)$ is convex as $\alpha>0$ and $r\leq \mbox{rank}(A)$, then by Lemma \ref{lemma-Jensen}, we have
$$
\begin{array}{ll}
\sum\limits_{ i=1}^{\mbox{rank}(A)}\frac{1}{\|W^{-1}\|^2_F\sigma_i(A)^{-2}+\alpha}&\geq \sum\limits_{ i=1}^{r}\frac{1}{\|W^{-1}\|^2_F\sigma_i(A)^{-2}+\alpha}
\\
&\geq \frac{r}{\alpha+\frac{1}{r}\sum\limits_{ i=1}^{r}\|W^{-1}\|^2_F\sigma_i(A)^{-2}   }.
\end{array}
$$
Thus
\begin{equation}
\label{proof-e1}
\lambda_{\min}(g)\geq\frac{-(k-1)}{\alpha}+\frac{r}{\alpha+\frac{1}{r}\sum\limits_{ i=1}^{r}\|W^{-1}\|^2_F\sigma_i(A)^{-2}   }.
\end{equation}
It is easily to check that
$$
\alpha_{\max}=\frac{\sqrt{k-1}}{\sqrt{r}-\sqrt{k-1}}\cdot \frac{\|W^{-1}\|^2_F}{r}\sum\limits_{ i=1}^{r}\sigma_i(A)^{-2} 
$$
 maximizes the right hand side of \eqref{proof-e1}. Then by simple calculation, we can obtain \eqref{proof-main2-1}.
\end{proof}
The proof of Theorem \ref{main-theorem} is very similar to that of Theorem \ref{main-theorem2}.
\begin{proof}[Proof of Theorem \ref{main-theorem}]
From the proof of Theorem \ref{main-theorem}, it suffices to prove that
\begin{equation}
\label{min-g}
\lambda_{\min}(g)\geq \bigg(1-\sqrt{\frac{k-1}{\mbox{srank}_4(A)}}\bigg)^2\frac{\|A\|^2_F}{\|W^{-1}\|^2_F}.
\end{equation}
If we write $k=(1-\epsilon)^2\mbox{srank}_4(A)+1$ where $\epsilon\in(0,1)$, then \eqref{min-g} is equivalent to
$$
\lambda_{\min}(g)\geq \epsilon^2\frac{\|A\|^2_F}{\|W^{-1}\|^2_F}.
$$
Note that the function $x\to \frac{1}{\|W^{-1}\|^2_F/\|A\|^{2}_F+\alpha x}$ on $[0,+\infty)$ is convex as $\alpha>0$ and then using Lemma \ref{lemma-Jensen}, we have
$$
\begin{array}{ll}
\sum\limits_{ i=1}^{\mbox{rank}(A)}\frac{1}{\|W^{-1}\|^2_F\sigma_i(A)^{-2}+\alpha}&= \sum\limits_{ i=1}^{\mbox{rank}(A)}\frac{\sigma_i(A)^{2}/\|A\|^{2}_F}{\|W^{-1}\|^2_F/\|A\|^{2}_F+\sigma_i(A)^{2}/\|A\|^{2}_F\alpha }
\\[5mm]
&\geq \frac{1}{\|W^{-1}\|^2_F/\|A\|^{2}_F+\sum\limits_{ i=1}^{\mbox{rank}(A)}\big(\sigma_i(A)^{2}/\|A\|^{2}_F\big)^2\alpha }
\\[5mm]
&=\frac{\mbox{srank}_4(A)}{\mbox{srank}_4(A)\|W^{-1}\|^2_F/\|A\|^{2}_F+\alpha }
\end{array}
$$
From \eqref{proof-111} we see that
\begin{equation}
\label{proof-main2-2}
\lambda_{\min}(g)
\geq\frac{-(k-1)}{\alpha}+\frac{\mbox{srank}_4(A)}{\mbox{srank}_4(A)\|W^{-1}\|^2_F/\|A\|^{2}_F+\alpha }.
\end{equation}
It is easily to check that
$$
\alpha_{\max}=\frac{\sqrt{k-1}}{\sqrt{\mbox{srank}_4(A)}-\sqrt{k-1}}\cdot \frac{\mbox{srank}_4(A)\|W^{-1}\|^2_F}{\|A\|^{2}_F} 
$$
 maximizes the right hand side of \eqref{proof-main2-2}. Then by simple calculation, we can obtain \eqref{min-g}.
\end{proof}

\section{A deterministic greedy selection algorithm}
\label{section-algorithm}
This section aims to present a deterministic greedy selection algorithm for the restricted invertibility, inspiring by the arguments in \cite{spielman17,xx2019}. 
Essentially, the algorithm produces the subset $\mathcal{S}$ by iteratively adding indices to it.  Suppose that at the $(j-1)$-th ($1\leq j\leq k$) iteration, we  already found a partial assignment $s_1,\ldots,s_{j-1}$ (it is empty when $j=1$). Then at the $j$-th iteration, the algorithm finds an index $s_{j}\in[m]\setminus \{s_1,\ldots,s_{j-1}\}$ such that $\lambda_{k}(f_{s_1,\ldots,s_{j}})\geq \lambda_{k}(f_{s_1,\ldots,s_{j-1}})$.

%
 By the definition of interlacing families, we know that the polynomial corresponding to a partial assignment $s_1,\ldots,s_j\in[m]^j$ is given by
$$
f_{s_1,\ldots,s_j}(x):=\mathbb{E}\det\bigg(xI-\sum\limits_{i=1}^ja_{s_i}a_{s_i}^T-\sum\limits_{i=j+1}^k{\bf r_ir_i}^T\bigg),
$$
where ${\bf r_{j+1}},\ldots,{\bf r_{k}}$ are i.i.d. and take the values $a_1,\ldots,a_m$ with probabilities $\frac{w_1^{-2}}{\|W^{-1}\|^2_F},\ldots,\frac{w_m^{-2}}{\|W^{-1}\|^2_F}$.
Now we want to find an index $s_{j+1}\in[m]$ such that $\lambda_k(f_{s_1,\ldots,s_{j+1}})\geq \lambda_k(f_{s_1,\ldots,s_j})$.
So one has to efficiently compute any partial assignment polynomial $f_{s_1,\ldots,s_j}$.

Let $C:=\sum\limits_{i=1}^ja_{s_i}a_{s_i}^T$ and $B:=\mathbb{E}{\bf r_ir_i}^T=AA^T/\|W^{-1}\|^2_F$ and applying Lemma \ref{lem-p1} repeatedly, we have
\begin{equation}
\label{fs}
\begin{array}{ll}
f_{s_1,\ldots,s_j}(x)&=\mathbb{E}\det\bigg[  xI-C-\sum\limits_{i=j+1}^k {\bf r_ir_i}^T\bigg]
\\
&=\bigg(\prod\limits_{i=j+1}^k (1-\partial_{z_i})\bigg)\det\bigg[  xI-C+\sum\limits_{i=j+1}^k\frac{z_i}{\|W^{-1}\|^2_F} AA^T\bigg]\bigg |_{z_{j+1}=\cdots=z_k=0}
\\
&=(1-\partial_{z})^{k-j}\det[  xI-C+zB] |_{z=0}.
\end{array}
\end{equation}
One can use the elementary symmetric function and fast polynomial interpolation to compute the bivariate polynomial $\det[xI-C+zB]$ in $O(n^{\theta+1})$ time, where $\theta\in (2,2.373)$ is the matrix multiplication complexity exponent \cite{keller85,Gall2014}.
Then applying the operator
$$
(1-\partial_z)^{k-j}=\sum\limits_{i=0}^{k-j}(-1)^{k-j-i}\left(
                                                          \begin{array}{c}
                                                            k-j \\
                                                            i \\
                                                          \end{array}
                                                        \right)\partial_z^i
$$
to each coefficient of $\det[xI-C+zB]$ and letting $z=0$,  which can be carried out in $O(n^2)$ time. Thus, we can compute $f_{s_1,\ldots,s_j}(x)$ in $O(n^{\theta+1})$ time.
One may refer to  \cite[Section 4.1]{spielman17} for more details.

Now, suppose that $
f_{s_1,\ldots,s_j}(x)$ has the following formula
$$
f_{s_1,\ldots,s_j}(x)=b^j_nx^n+b^j_{n-1}x^{n-1}+\cdots+b^j_0.
$$
 It follows from  \cite{root1994}  that  finding the roots of $f_{s_1,\ldots,s_j}(x)$ is equivalent to finding the eigenvalues of the upper Hessenberg matrix
$$
\label{H_j}
H_j=\left(
  \begin{array}{ccccc}
    -\frac{b^j_{n-1}}{b^j_n} & -\frac{b^j_{n-2}}{b^j_{n-1}} & -\frac{b^j_{n-3}}{b^j_{n-2}} & \cdots & -\frac{b^j_{1}}{b^j_0} \\
    1 & 0 & 0 & \cdots & 0 \\
    0 & 1 & 0 & \cdots & 0 \\
    \vdots & \vdots & \vdots & \vdots & \vdots \\
    0 & 0 & \cdots & 1 & 0 \\
  \end{array}
\right).
$$
Thus we can compute the $k$-th root of
$f_{s_1,\ldots,s_j}(x)$ by a number of methods, such as doing QR decomposition for $H_j$.

The deterministic greedy selection algorithm is stated as follows:
\begin{algorithm}[H]
\caption{A deterministic greedy selection algorithm}
\begin{algorithmic}[H]
\Require
$A\in\mathbb{R}^{n\times m}$ of rank $n$;  sampling parameter $k\in \{1,\ldots,\mbox{rank}(A)\}$.
\begin{enumerate}
\item[1:] Set $s_0=\emptyset$, and $j:=1$.
\item[2:]  For each $s \in[m]\setminus \{s_1,\ldots,s_{j-1}\}$, compute the polynomial $f_{s_1,\ldots,s_{j-1},s}(x)$.
\item[3:] Using the QR decomposition algorithm, for each $s \in[m]\setminus \{s_1,\ldots,s_{j-1}\}$, compute the $k$-th root of $f_{s_1,\ldots,s_{j-1},s}(x)$.
\item[4:] Find
$$
s_j=
\argmax{s\in[m]\setminus\{s_1,\ldots,s_{j-1}\}}\lambda_{k}\big(f_{s_1,\ldots,s_{j-1},s}(x)\big).
$$
\item[5:] If $j>k$, stop the algorithm. Otherwise, set $j=j+1$ and return to Step $3$.
\end{enumerate}

\Ensure
Subset $\cS=\{s_1,\ldots,s_k\}.$
\end{algorithmic}
\end{algorithm}

We have the following theorem for Algorithm $1$.

\begin{theorem}\label{th:alg}
For any $k\leq \mbox{rank}(A)$. Algorithm $1$ can output a subset $\cS=\{s_1,\ldots,s_k\}$ such that
\begin{equation}
\label{algbound}
\sigma_{\min}(A_{\cS}W_{\cS}
)^2\geq \lambda_{k}(f_{\emptyset}(x)).
\end{equation}
The running time complexity is $O\big(k(m-\frac{k}{2})n^{\theta+1}\big)$ where $\theta\in (2,2.373)$ is the matrix multiplication complexity exponent.
\end{theorem}

By the proof of Theorems \ref{main-theorem2} and \ref{main-theorem}, it is easy to find that \eqref{algbound} implies the bound established in those theorems.

\begin{proof}[Proof of Theorem \ref{th:alg}]
By Step $4$ in Algorithm $1$ and the definition of interlacing families, we know that
$$
\lambda_{k}\big(f_{s_1,\ldots,s_{k}}(x)\big)\geq
\lambda_{k}\big(f_{s_1,\ldots,s_{k-1}}(x)\big)\geq
\cdots\geq \lambda_{k}\big(f_{s_1}(x)\big)\geq  \lambda_{k}\big(f_{\emptyset}(x)\big).
$$
We next establish the  running time complexity.

The main cost of Algorithm $1$ is  Steps $2$ and $3$. In Step $2$, at the $j$-th iteration, by \eqref{fs} we know that $f_{s_1,\ldots,s_{j-1},s}(x)$ can be computed in $O(n^{\theta+1})$ time for any fixed $s \in[m]\setminus \{s_1,\ldots,s_{j-1}\}$. Therefore, the total computational cost of Step $2$ is  $O((m-j+1)n^{\theta+1})$. As the computational cost of the QR decomposition is $O(n^3)$, so the time complexity for computing the $k$-th root of $f_{s_1,\ldots,s_{j-1},s}(x)$ over all $s \in[m]\setminus \{s_1,\ldots,s_{j-1}\}$ is $O((m-j+1)n^3)$. Thus, Algorithm $1$ produces the subset $\cS$ in $O\big(k(m-\frac{k}{2})n^{\theta+1}\big)$ time.
\end{proof}


\subsection*{Acknowledgements}
The author would like to thank Prof. Zhiqiang Xu at AMSS, Chinese Academy of Sciences for many useful discussions.


\end{document}